\documentclass[10pt]{amsart}
\usepackage[latin1]{inputenc}
\usepackage[T1]{fontenc}
\usepackage[english]{babel}

\usepackage{amsmath,amssymb,amsthm,amsfonts}

 \newtheorem{theo}{Theorem}[section]
 
 \newtheorem{lem}[theo]{Lemma}
 \newtheorem{cor}[theo]{Corollary}

\newcommand{\loglike}[1]{\mathop{\rm #1}\nolimits}

\newcommand{\cone}{\loglike{cone}}
\newcommand{\eps}{\varepsilon}

\newcommand{\gam}{\gamma}

\newcommand{\cconv}{\overline{\mathrm{conv}}}

\newcommand{\lin}{\loglike{lin}}

\newcommand{\R}{{\mathbb{R}}}

\newcommand{\iy}{\infty}

 \renewcommand{\le}{\leqslant}
\renewcommand{\ge}{\geqslant}

\begin{document}

\begin{center}\small
[to appear in \emph{J.\ Math.\ Anal.\ Appl.} with doi \texttt{10.1016/j.jmaa.2012.06.031}]
\end{center}

\title[Extension of isometries between unit spheres]
  {Extension of isometries between unit spheres of finite-dimensional polyhedral Banach spaces}

\author{Vladimir Kadets}

\address[V.~Kadets]{Department of Mechanics and Mathematics\\ Kharkov V.N. Karazin National University\\
 pl. Svobody 4 \\ 61077 Kharkov, Ukraine}
\email{vova1kadets@yahoo.com}

\author{Miguel Mart\'{\i}n }

\address[M.~Mart\'{\i}n]{Departamento de An\'{a}lisis Matem\'{a}tico \\ Universidad de Granada \\ 18071 Granada, Spain}
\email{mmartins@ugr.es}

\thanks{The work of the first-named author
was partially supported by Junta de Andaluc\'{\i}a and FEDER grants FQM-185 and P09-FQM-4911 and by the program GENIL-PRIE of the CEI of the University of Granada. The second author was partially
supported by Spanish MICINN and FEDER project no.\
MTM2009-07498 and Junta de Andaluc\'{\i}a and FEDER grants FQM-185
and P09-FQM-4911.}

\date{June 20th, 2012}

\begin{abstract}
We prove that an onto isometry between unit spheres of
finite-dimensional polyhedral Banach spaces extends to a linear
isometry  of the corresponding spaces.
\end{abstract}
\maketitle

\section{Introduction}

In 1987, D.~Tingley proposed the following question
\cite{ting}: let $f$ be a bijective isometry between the unit
spheres $S_X$ and $S_E$ of real Banach spaces $X$, $E$
respectively. Is it true that $f$ extends to a linear
(bijective) isometry $F: X \longrightarrow E$ of the
corresponding spaces? Let us mention that this is equivalent to
the fact that the natural (positive) homogeneous extension of
$f$ (see \eqref{hom-extension}) is linear. He proved a useful
partial result:

\begin{theo}[Tingley's theorem \cite{ting}]\label{theo:tingley} If $X$ and $E$ are
finite-dimensional Banach spaces and $f: S_X \longrightarrow
S_E$ is a bijective isometry, then $f(-x) = - f(x)$ for all
$x\in S_X$.
\end{theo}

We recall that the classical Mazur-Ulam theorem  states that
every surjective isometry between $X$ and $E$ is affine and
that there is a result by P.~Mankiewicz \cite{Mankiewicz} which
states that every bijective isometry between convex bodies of
$X$ and $E$ can be uniquely extended to an affine isometry from
$X$ and $E$.

There is a number of publications devoted to Tingley's problem
(see \cite{ding} for a survey of corresponding results) and, in
particular, the problem is solved in positive for many concrete
classical Banach spaces. Surprisingly, the question for general
spaces remains open, even in dimension two.

Recently, L.~Cheng and Y.~Dong \cite{cheng} attacked the
problem for the class of polyhedral spaces (i.e.\ for those
spaces whose unit sphere is a polyhedron). Unfortunately their
interesting attempt failed by a mistake at the very end of the
proof. The authors told to us in a private communication that
they don't see how their proof can be repaired.

In this paper we present a new approach to Tingley's problem
that enables us to save partially the Cheng-Dong result.
Namely, we answer the problem in positive for
\textbf{finite-dimensional polyhedral spaces}. The idea of the
proof is to study the differentiability properties of $f$ and
of its homogeneous extension $F$. Although our main result is
about polyhedral spaces, for the sake of possible applications,
the technical differentiability lemmas are proved for general
finite-dimensional normed spaces.

\section{Notation}

Throughout the paper $X$, $E$ are $m$-dimensional Banach spaces
over the field of reals, $X^*$, $E^*$ are their dual spaces,
$S_X$, $B_X$, stand for the unit sphere and unit ball of the
corresponding space, $f: S_X \longrightarrow S_E$ is a
bijective isometry and, finally, $F: X \longrightarrow E$ is
the natural (positively) homogeneous extension of $f$, that is,
\begin{equation}\label{hom-extension}
F(0) = 0,\qquad F(x) = \|x\|\,f\left(x / \|x\|\right)\quad \bigl(x\in X\setminus\{0\}\bigr).
\end{equation}
Recall that, thanks to Tingley's theorem \ref{theo:tingley},
$F(-x)=-F(x)$ for every $x\in X$, so $F$ is homogeneous for the
negative scalars as well.

We will use the notation $\rho(x,y)=\|x-y\|$ for the metric in
both $S_X$ and $S_E$. We will use the notations $x^*(x)$ and
$\langle x^*,x\rangle$ to denote the action of $x^*\in X^*$ on
$x\in X$, and we also use the same notations for the action of
elements of $E^*$ on elements of $E$.

For every $A \subset X$, we denote by $\cone(A) = \{t x \,:\, x
\in A,\, t \ge 0\}$ the cone generated by $A$. For every $x \in
S_X$, we denote by $\gimel(x) \subset X^*$ the nonempty set of
support functionals of $x$, i.e.\ those $x^*\in X^*$ such that
$\|x^*\|=x^*(x)=1$. If $\gimel(x)$ consists of only one
element, we say that $x$ is a smooth point and the set of
smooth points of $S_X$ is denoted by $\Sigma(X)$. If $x \in
\Sigma(X)$, we denote the unique element of $\gimel(x)$ as
$\gamma(x)$. Recall that in finite-dimensional spaces, every
smooth point of the unit sphere is actually a Fr\'{e}chet
differentiability point for the map $x \longmapsto \|x\|$. This
means that for $x \in \Sigma(X)$, there is a function
$\eps_x(r)$ such that
\begin{equation} \label{smooth}
\frac{\eps_x(r)}{r} \underset{r \to
0}{\longrightarrow} 0 \quad \text{ and } \quad
\langle \gamma(x), z \rangle \le \|z\| \le (1 + \eps_x(r))\langle \gamma(x), z \rangle
\end{equation}
for every $z \in \cone(x + rB_X)$. For this and other standard
facts from convex geometry we refer to Rockafellar's book
\cite{rock}. Remark, that in the most valuable for us case of
polyhedral spaces, $x \in \Sigma(X)$ if and only if $x$ is an
interior point of an $(m-1)$-dimensional face and $\eps_x(r) =
0$ for sufficiently small $r$.

\section{The differentiability lemmas}

\begin{lem}\label{lem 1}
Let $x, y, y_n \in S_X$, $x \neq y$, such that $\frac{x -
y}{\|x - y\|} \in \Sigma(X)$ and suppose that
$$
y_n \longrightarrow y, \,\,\, \frac{y - y_n}{\|y - y_n\|} \longrightarrow u
 \qquad \mathrm{ as } \,\,\, n \to \infty.
$$
Then
\begin{equation} \label{lem 1 eq}
\frac{\rho(x, y_n) -  \rho(x, y)}{\rho(y, y_n)} \longrightarrow \left\langle \gamma
\left(\frac{x - y}{\|x - y\|}\right), u \right\rangle \,\, \mathrm{ as } \,\,\, n \to \infty.
\end{equation}
\end{lem}

\begin{proof}
If we denote $r_n = \|y - y_n\|/\|x - y\|$
then
$$
\|(x - y_n) - (x - y)\| = r_n\|x - y\|,
$$
i.e.
$$
x - y_n \in \cone\left(\frac{x - y}{\|x - y\|} + r_n B_X\right),
$$
and we can use \eqref{smooth} to get
\begin{multline*}
\frac{\left\langle \gamma \left(\frac{x - y}{\|x -
y\|}\right), x - y_n \right\rangle - \|x - y\|}{\|y - y_n\|}
 \le \frac{\rho(x, y_n) -  \rho(x, y)}{\rho(y, y_n)} \\
 \le \frac{(1 + \eps_{\frac{x - y}{\|x - y\|}}(r_n))\left\langle \gamma \left(\frac{x - y}{\|x - y\|}\right),
x - y_n \right\rangle -  \|x - y\|}{\|y - y_n\|}.
\end{multline*}
Since $\|x - y\| = \left\langle \gamma \left(\frac{x - y}{\|x -
y\|}\right), x - y \right\rangle$, we can continue as follows:
\begin{multline*}
\left\langle \gamma \left(\frac{x - y}{\|x - y\|}\right),
\frac{y - y_n}{\|y - y_n\|} \right\rangle \le \frac{\rho(x,
y_n) -  \rho(x, y)}{\rho(y, y_n)}
 \\
\le \left\langle \gamma \left(\frac{x - y}{\|x - y\|}\right), \frac{y - y_n}{\|y - y_n\|} \right\rangle +
\frac{\eps_{\frac{x - y}{\|x - y\|}}(r_n)}{r_n\|x - y\|}
\left\langle \gamma \left(\frac{x - y}{\|x - y\|}\right), x - y_n \right\rangle.
\end{multline*}
Passing to limit when $n \to \infty$, we get the desired
result.
\end{proof}

For $y\in S_X$, we write $D_y = \{x \in S_X\,:\, \|x + y\| <
2\}$, which is a relatively open subset of $S_X$, and observe
that $D_y$ consists of those points of the sphere for which the
line interval $]x, y[ = \{\lambda x + (1 - \lambda) y \,:\, 0 <
\lambda < 1\}$ lies in the open unit ball. Also observe that
$D_y = \{x \in S_X\,:\, \rho(-y, x) < 2\}$ so, thanks to
Tingley's Theorem \ref{theo:tingley}, $f$ maps bijectively
$D_y$ onto $D_{f(y)}$. We denote by $W_y$ the set of those $x
\in D_y$ for which
$$
\frac{x - y}{\|x - y\|} \in \Sigma(X) \quad \text{ and }
 \quad \frac{f(x) - f(y)}{\|f(x) - f(y)\|} \in \Sigma(E).
$$

\begin{lem}\label{lem 2+}
$D_y \setminus W_y$ is negligible in $D_y$ so, in particular,
$W_y$ is dense in $D_y$.
\end{lem}

\begin{proof}
Consider the function $g: D_y \longrightarrow S_X$, $g(x) =
\frac{x - y}{\|x - y\|}$ for every $x\in D_y$. Then, $g$ is
injective, $g(D_y)$ is relatively open, and $g$, as well as
$g^{-1}$ are locally Lipschitz. Since  $S_X \setminus
\Sigma(X)$ is negligible in $S_X$, $g^{-1}(S_X \setminus
\Sigma(X))$ is negligible in $D_y$, i.e.\ the set $\left\{x \in
D_y\,:\, \frac{x - y}{\|x - y\|} \notin \Sigma(X)\right\}$ is
negligible in $D_y$. Analogously, from the fact that $S_E
\setminus \Sigma(E)$ is negligible in $S_E$, we deduce that the
set $\left\{x \in D_y\,:\, \frac{f(x) - f(y)}{\|f(x) - f(y)\|}
\notin \Sigma(E)\right\}$ is negligible in $D_y$. Finally,
$D_y\setminus W_y$ is the union of two negligible sets.
\end{proof}

We say that a subset $A$ of the unit sphere of the dual of a
Banach space $Z$ is \emph{total} if for every $z\in Z$, there
is $z^*\in A$ such that $z^*(z)\neq 0$. The set $A$ is said to
be \emph{1-norming} if $\sup\{|z^*(z)|\,:\, z^*\in A\}=\|z\|$
for every $z\in Z$.

\begin{lem}\label{lem 3}
For every $y \in S_X$ the set
$$
\left\{\gamma\left(\frac{x - y}{\|x - y\|}\right): x \in W_y\right\}
$$
is total over $X$, and
$$
\left\{\gamma\left(\frac{f(x) - f(y)}{\|f(x) - f(y)\|}\right): x \in W_y\right\}
$$
is total over $E$. Moreover, if $y \in \Sigma(X)$ (resp.\ $f(y)
\in \Sigma(E)$), then the corresponding set is 1-norming.
\end{lem}

\begin{proof}
Let us start with the ``moreover'' part. If $y$ is a smooth
point of $S_X$, then
$$
\left\{\frac{x - y}{\|x - y\|}\,:\, x
\in D_y\right\} \supset \{z \in S_X\,:\, \langle\gam(y), z\rangle <
0\},
$$
i.e.\ it contains the intersection of the sphere with an open
half-space. This together with the density of $W_y$ in $D_y$
makes the ``moreover'' part evident.

For the main part of the statement, denote by $A$ the relative
interior in $S_X$ of the set $\left\{\frac{x - y}{\|x -
y\|}\,:\, x \in D_y\right\}$. Since $\left\{\frac{x - y}{\|x -
y\|}: x \in W_y\right\}$ is dense in $A$,
$$
\cconv\left\{\gamma\left(\frac{x - y}{\|x - y\|}\right): x \in W_y\right\} \supset \bigcup_{a \in A}\gimel(a).
$$
So it is sufficient to show that for every $z \in X$ there is
$a \in A$ and $x^* \in \gimel(a)$ such that $x^*(z) \neq 0$.
 Consider the two-dimensional subspace $Z \subset X$ spanned by $y$ and $z$.
If $y$ is a smooth point of $S_Z$, then the job is done by the
same reason as in the ``moreover'' part. If $y$ is not a smooth
point of $S_Z$, then $a = -y \in A$ is not a smooth point of
$S_Z$ neither, so at least one of support functionals in this
point $a$ must take a non-zero value at $z$.

The same argument works for the set $
\left\{\gamma\left(\frac{f(x) - f(y)}{\|f(x) -
f(y)\|}\right)\,:\, x \in W_y\right\}$.
\end{proof}

\begin{lem}\label{lem 4}
For every $y \in S_X$ and for every sequence $(y_n)$ on $S_X$
converging to $y$, if the sequence $\left(\frac{y - y_n}{\|y -
y_n\|}\right)$ is convergent, then so is the sequence
$\left(\frac{f(y) - f(y_n)}{\|f(y) - f(y_n)\|}\right)$.
Moreover, for every $x \in W_y$
\begin{equation} \label{lem 4 eq}\textstyle
 \left\langle \gamma \left(\frac{x - y}{\|x - y\|}\right),
\lim\limits_{n \to \infty}\frac{y - y_n}{\|y - y_n\|}
\right\rangle
 = \left\langle \gamma \left(\frac{f(x) - f(y)}{\|f(x) -
f(y)\|}\right),
 \lim\limits_{n \to \infty}\frac{f(y) - f(y_n)}{\|f(y) - f(y_n)\|} \right\rangle
\end{equation}
\end{lem}

\begin{proof}
Denote $u = \lim\limits_{n \to \infty}\frac{y - y_n}{\|y -
y_n\|}$. Assume at first that $\lim\limits_{n \to
\infty}\frac{f(y) - f(y_n)}{\|f(y) - f(y_n)\|}$ exists, and
denote it $v$. Then, according to Lemma~\ref{lem 1}, we have
\begin{align*}
\left\langle \gamma \left(\frac{x - y}{\|x - y\|}\right), u \right\rangle &=
\lim_{n \to \infty}\frac{\rho(x, y_n) -  \rho(x, y)}{\rho(y, y_n)}
\\ &
= \lim_{n \to \infty}\frac{\rho(f(x), f(y_n)) -  \rho(f(x), f(y))}{\rho(f(y), f(y_n))} \\ & =
\left\langle \gamma \left(\frac{f(x) - f(y)}{\|f(x) - f(y)\|}\right), v \right\rangle.
\end{align*}
This proves \eqref{lem 4 eq}. Now, assume that $v_1, v_2$ are
limits of some subsequences of the sequence $\left(\frac{f(y) -
f(y_n)}{\|f(y) - f(y_n)\|}\right)$. Applying for these
subsequences the already proved condition \eqref{lem 4 eq} we
get that
$$
\left\langle \gamma \left(\frac{f(x) - f(y)}{\|f(x) - f(y)\|}\right), v_1 \right\rangle =
\left\langle \gamma \left(\frac{f(x) - f(y)}{\|f(x) - f(y)\|}\right), v_2 \right\rangle
$$
for all $x \in W_y$. By Lemma~\ref{lem 3} this means that $v_1
= v_2$.
\end{proof}

For every $y \in S_X$, we write $\Lambda_y$ to denote the set
of all limiting points of the expression
$$
\frac{y - z}{\|y - z\|}
$$
when $z \to y$, $z \in S_X$ ($\Lambda_y$ is the set of tangent
directions) and we observe that
\begin{enumerate}
\item[(a)] if $y \in \Sigma(X)$, then $\Lambda_y = S_{\ker
    \gamma(y)}$, i.e.\ it is the unit sphere of a
    hyperplane,
\item[(b)] otherwise, $\Lambda_y$ is the intersection of
    the unit sphere with the boundary of the supporting
    cone $\{u \in X\,:\, x^*(u) \ge 0 \ \forall x^* \in
    \gimel(y)\}$ and, in particular, $\lin \Lambda_y=X$.
\end{enumerate}
Let us also observe that Lemma~\ref{lem 4} means that the
correspondence
$$
\lim_{n \to \infty}\frac{y - y_n}{\|y - y_n\|} \longrightarrow \lim_{n \to \infty}\frac{f(y) - f(y_n)}{\|f(y) - f(y_n)\|}
$$
defines a bijective map between $\Lambda_y$ and
$\Lambda_{f(y)}$. We write $F_y: \Lambda_y \longrightarrow
\Lambda_{f(y)}$ for this map. With this notation we can rewrite
\eqref{lem 4 eq} as follows: for every $x \in W_y$, $u \in
\Lambda_y$
\begin{equation} \label{lem 4 eq+}
 \left\langle \gamma \left(\frac{x - y}{\|x - y\|}\right), u \right\rangle
 = \left\langle \gamma \left(\frac{f(x) - f(y)}{\|f(x) - f(y)\|}\right), F_y(u) \right\rangle
\end{equation}

\begin{lem}\label{lem 5}
The map $F_y$ extends to a linear isomorphism between $\lin
\Lambda_y$ and $\lin \Lambda_{f(y)}$ (we will denote this
extension again by $F_y$). Moreover, if $y \in \Sigma(X)$ and
$f(y) \in \Sigma(E)$, then this linear isomorphism is an
isometry.
\end{lem}

\begin{proof}
Let $v_1, \ldots, v_N \in \Lambda_y$ and $a_1, \ldots , a_N \in
\R$. By Lemma~\ref{lem 3}, the set
$$
\left\{\gamma\left(\frac{f(x) - f(y)}{\|f(x) - f(y)\|}\right): x \in W_y\right\}
$$
is total over $E$. Since $\dim E < \iy$, this set of
functionals is norming with some constant $C>0$. Therefore,
\begin{align*}
\left\|\sum_{j=1}^N a_j F_y(v_j)\right\| & \le C \sup\left\{\left|\left\langle \gamma\left(\frac{f(x) - f(y)}{\|f(x) - f(y)\|}\right),
\sum_{j=1}^N a_j F_y(v_j) \right\rangle\right| : x \in W_y\right\}
\\ &
= C \sup\left\{\left|\sum_{j=1}^N a_j \left \langle \gamma\left(\frac{f(x) - f(y)}{\|f(x) - f(y)\|}\right),
 F_y(v_j) \right\rangle\right|: x \in W_y \right\}
\\ &
= C \sup\left\{\left|\sum_{j=1}^N a_j\left \langle \gamma\left(\frac{x - y}{\|x - y\|}\right),
 v_j \right\rangle\right|: x \in W_y \right\}
\\ &
= C \sup\left\{\left|\left \langle \gamma\left(\frac{x - y}{\|x - y\|}\right),
\sum_{j=1}^N a_j v_j \right\rangle\right|: x \in W_y\right\}
\\ &
\le C\, \left\|\sum_{j=1}^N a_j v_j\right\|.
\end{align*}
This demonstrates the possibility of a linear extension and we
may interchange the rolles of $X$ and $E$ to get the reversed
inequality and so an isomorphism. The ``moreover'' part follows
from the ``moreover'' part of Lemma~\ref{lem 3} since, in such
a case, $C=1$.
\end{proof}

The next goal is to study what happens with the supporting
functionals in a non-smooth point $y \in S_X$.

\begin{lem}\label{lem 6}
Let $(x_n)$ be a sequence in $W_y$ such that $(x_n)
\longrightarrow -y$. Assume that $\gamma(\frac{x_n - y}{\|x_n -
y\|}) \longrightarrow y^* \in \gimel(-y)$. Then, there exists
$e^*:=\lim\limits_{n\to \infty} \gamma\left(\frac{f(x_n) -
f(y)}{\|f(x_n) - f(y)\|}\right)$ and
\begin{equation} \label{lem dual eq}
  \left\langle y^*, u \right\rangle
 = \left\langle e^*, F_y(u) \right\rangle
\end{equation}
for every $u \in \Lambda_y$.
\end{lem}

\begin{proof}
Denote $z_n = \frac{f(x_n) - f(y)}{\|f(x_n) - f(y)\|}$. At
first assume that $\lim\limits_{n\to \infty} \gamma(z_n) =:
e^*$ exists, then \eqref{lem dual eq} is just a limiting case
of \eqref{lem 4 eq+}. Now suppose that $e_1^*$ and $e_2^*$ are
limits of some subsequences of $\bigl(\gamma(z_n)\bigr)$. Then
\eqref{lem dual eq} is valid for both $e_1^*, e_2^*$, so for
every $u \in \Lambda_y$
$$
\left\langle e_1^* - e_2^*, F_y(u) \right\rangle = 0.
$$
Also, evidently, $e_1^*(f(y)) = e_2^*(f(y)) =-1$, so $e_1^* -
e_2^* \in \bigl[\Lambda_{f(y)} \cup \{f(y)\}\bigr]^\bot =
\{0\}$.
\end{proof}

Denote by $M_y^*$ the set of elements in $S_{X^*}$ of the form
$\lim\limits_{n\to\infty}\gamma(\frac{x_n - y}{\|x_n - y\|})$,
where $(x_n)$ is a sequence in $W_y$ converging to $-y$ and
observe that $M_y^*\subset \gimel(-y)$. We write
$M_{f(y)}^*\subset \gimel(-f(y))$  for the set of elements in
$S_{E^*}$ of the form
$\lim\limits_{n\to\infty}\gamma(\frac{f(x_n) - f(y)}{\|f(x_n) -
f(y)\|})$, where $(x_n)$ is a sequence in $W_y$ converging to
$-y$. Equivalently, $M_{f(y)}^*$ is the set of elements in
$S_{E^*}$ of the form $\lim\limits_{n\to\infty}\gamma(\frac{z_n
- f(y)}{\|z_n - f(y)\|})$, where $(z_n)$ is a sequence in
$S_{E}$ converging to $-f(y)$ such that $\|z_n-f(y)\|<2$,
$z_n\in \Sigma(E)$ and $f^{-1}(z_n)\in \Sigma(X)$.

In the same way as in the definition of $F_y$, we can now
define a bijective map $G_y: M_y^* \longrightarrow M_{f(y)}^*$
by
$$
G_y\left(\lim_{n\to\infty}\gamma\left(\frac{x_n - y}{\|x_n - y\|}\right)\right)
:= \lim_{n\to\infty}\gamma\left(\frac{f(x_n) - f(y)}{\|f(x_n) - f(y)\|}\right).
$$
Then \eqref{lem dual eq} can be re-written as
\begin{equation} \label{lem dual eq+}
  \left\langle y^*, u \right\rangle
 = \left\langle G_y(y^*), F_y(u) \right\rangle
\end{equation}
for every $y^* \in  M_y^*$ and for every $u \in \lin
\Lambda_y$. Now, as in Lemma~\ref{lem 5} and taking into
account that the closed convex hull of $M_y^*$ equals to
$\gimel(-y) = - \gimel(y)$, we can deduce the following.

\begin{lem}\label{lem 5++}
$G_y$ extends to a linear isomorphism between $\lin \gimel(y)$
and $\lin \gimel(f(y))$ (we will denote this extension again as
$G_y$) satisfying that $G_y(\gimel(y))=\gimel(f(y))$ and
$$
\left\langle y^*, u \right\rangle = \left\langle G_y y^*, F_y u \right\rangle
$$
for all $y^* \in \lin \gimel(y)$, $u \in \lin \Lambda_y$.
Therefore, $\dim \lin \gimel(y) = \dim \lin \gimel(f(y))$ and,
in particular, $f$ maps smooth points into smooth points.
\end{lem}

\begin{proof}
Recall first that outside of \eqref{lem dual eq+} we know that
$\left\langle y^*, y \right\rangle   = \left\langle G_y(y^*),
f(y) \right\rangle = -1$ for every $y^* \in  M_y^*$. Let
$v_1^*, \ldots, v_N^* \in M_y^*$, $a_1, \ldots , a_N \in \R$.
The set $\Lambda_{f(y)} \cup \{f(y)\}$ spans all the E, which
means, thanks to the finite-dimensionality of $E$, that this
set is norming for $E^*$ with some constant $C>0$. So, writing
$\vee$ to denote the maximum of two numbers, we have
\begin{align*}
 & \textstyle \hspace*{-0.5cm} \left\|\sum\limits_{j=1}^N a_j G_y(v_j^*)\right\| \\ & \textstyle \le
C \left(\sup\left\{\left|\left\langle \sum\limits_{j=1}^N a_j G_y(v_j^*),
 F_y(u) \right\rangle\right| : u \in \Lambda_{y}\right\}  {\bigvee}
\left|\left\langle \sum\limits_{j=1}^N a_j G_y(v_j^*),  f(y) \right\rangle\right|\right)
\\ & \textstyle
= C \left(\sup\left\{\left|\sum\limits_{j=1}^N a_j\left\langle  G_y(v_j^*),
F_y(u) \right\rangle\right| : u \in \Lambda_{y}\right\}  {\bigvee}
\left|\sum\limits_{j=1}^N a_j\left\langle  G_y(v_j^*),  f(y)y \right\rangle\right|\right)
\\ & \textstyle
= C \left(\sup\left\{\left|\sum\limits_{j=1}^N a_j\left\langle  v_j^*,
u \right\rangle\right| : u \in \Lambda_{y}\right\}  {\bigvee}
\left|\sum\limits_{j=1}^N a_j\left\langle v_j^*,  y
\right\rangle\right|\right)
\\ & \textstyle
= C \left(\sup\left\{\left|\left\langle  \sum\limits_{j=1}^N a_jv_j^*,
u \right\rangle\right| : u \in \Lambda_{y}\right\}
{\bigvee} \left|\left\langle \sum\limits_{j=1}^N a_jv_j^*,  y \right\rangle\right|\right)
\\ & \textstyle
\le C \left\|\sum\limits_{j=1}^N a_jv_j^*\right\|.
\end{align*}
This demonstrates the possibility of linear extension.
\end{proof}

\section{The main results}

Recall that $F$ stands for the homogeneous extension of $f$,
see \eqref{hom-extension}. We denote
$$
[F'(y)](z) = \lim_{a \to 0^+}\frac1a \left(F(y + a z) - F(y)\right)
$$
the derivative of $F$ at point $y$ in direction $z$. This is
just the first step in the definition of the Gateaux
differential: $F$ is Gateaux differentiable if $[F'(y)](z)$
depends on $z$ linearly and continuously. In the
finite-dimensional case, continuity follows from linearity. We
also denote $H(y,z) \subset \gimel(y)$ the set of all $y^* \in
\gimel(y)$ such that
\begin{equation} \label{y*z}
\lim_{a \to 0^+} \frac1a\left(\|y + az\| - 1\right)= y^*(z)
\end{equation}
and observe that $H(y,z) \neq \emptyset$ by the convexity of
the norm.

\begin{lem}\label{lem 7}
For every $y \in S_X$, $z \in X$, $y^* \in H(y,z)$, we have
$z-\langle y^*,z\rangle y \in \lin \Lambda_y$ and
$$
[F'(y)](z) = \langle y^*, z\rangle f(y) + F_y\left(z - \langle y^*, z\rangle y\right).
$$
\end{lem}

\begin{proof}
Observe that
$$
\lim_{a \to 0^+}\frac1a \left( \frac{y + az}{\|y + az\|} - y\right)=
\lim_{a \to 0^+}\frac1a \left( y + az - \|y + az\|y\right)= z - y^*(z)y,
$$
and denote
$$
u := \lim_{a \to 0^+}\frac{ \frac{y + az}{\|y + az\|} - y}{\left\|\frac{y + az}{\|y + az\|} - y \right\|}
= \frac{z - y^*(z)y}{\|z - y^*(z)y\|}
$$
which, evidently, belongs to $\Lambda_y$. Now we can calculate
the limit that we need as follows:
\begin{align*}
[F'(y)](z) & = \lim_{a \to 0^+}\frac1a\left(\|y + az\|
f\left(\frac{y + az}{\|y + az\|}\right) -f(y)  \right)
\\ &
= \lim_{a \to 0^+}\frac1a\left(\|y + az\| - 1\right)f\left(\frac{y + az}{\|y + az\|}\right)
 \\ & \qquad \qquad +
\lim_{a \to 0^+}\frac1a\left( f\left(\frac{y + az}{\|y + az\|}\right) -f(y)\right)
\\ &
= y^*(z)f(y) + \lim_{a \to 0^+}\frac1a \left\|\frac{y + az}{\|y + az\|} - y \right\| \cdot
\lim_{a \to 0^+} \frac{ f\left(\frac{y + az}{\|y + az\|}\right) - f(y)}{\left\| f\left(\frac{y + az}{\|y + az\|}\right) - f(y) \right\|}
\\ &
= y^*(z)f(y) + \|z - y^*(z)y\| F_y(u) = y^*(z)f(y) + F_y(z - y^*(z)y). \qedhere
\end{align*}
\end{proof}

We are now ready to present the most important results of the
paper. The first one contains two sufficient conditions
assuring the differentiability of $F$.

\begin{theo} \label{diff-cond}
In the following cases we can guaranty the Gateaux
differentiability of $F$ in the point $y \in S_X$:
\begin{enumerate}
\item if $y \in \Sigma(X)$,
\item if $\lin \gimel(y) = X^*$.
\end{enumerate}
\end{theo}

\begin{proof}
(1). If $y \in \Sigma(X)$, then $H(y,z) = \{\gam(y)\}$,
$$
[F'(y)](z) = \langle\gamma(y), z\rangle f(y) + F_y\left(z - \langle\gamma(y), z\rangle y\right),
$$
so it linearly depends on $z$. \vspace{3 mm}

(2). In this case $y$ is not a smooth point, so $\lin \Lambda_y
= X$, and $F_y(y)$ is correctly defined. Let us prove that
$f(y) - F_y(y) = 0$. In fact, according to Lemma~\ref{lem 5++},
$\dim \lin \gimel(y) = \dim \lin \gimel(f(y))$, consequently
$\lin \gimel(f(y)) = E^*$. This implies that it is sufficient
to show that $\langle G_y(y^*), f(y) - F_y(y)\rangle = 0$ for
all $y^* \in \gimel(y)$. In fact, according to the same
Lemma~\ref{lem 5++}
$$
\langle G_y(y^*), f(y) - F_y(y)\rangle = \langle G_y(y^*), f(y)\rangle - \langle G_y(y^*), F_y(y)\rangle
$$
$$
1 - \langle y^*, y \rangle = 0.
$$
Now, fix $x^* \in \gimel(y)$ and let us show that for every $z
\in X$
\begin{equation} \label{previous}
[F'(y)](z) = \langle x^*, z\rangle f(y) + F_y\left(z - \langle x^*, z\rangle y\right).
\end{equation}
This will give us the linearity of $[F'(y)](z)$ in the variable
$z$. Let us check \eqref{previous}. According to Lemma~\ref{lem
7} for $y^* \in H(y,z)$ we have the representation
$$
[F'(y)](z) = \langle y^*, z\rangle f(y) + F_y\left(z - \langle y^*, z\rangle y\right).
$$
Let us compare this with \eqref{previous}:
\begin{align*}
& \left(\langle y^*, z\rangle f(y) + F_y\left(z - \langle y^*, z\rangle y\right)\right) -
\left(\langle x^*, z\rangle f(y) + F_y\left(z - \langle x^*, z\rangle y\right)\right)
\\ & = \langle y^* - x^*, z\rangle \left(f(y) - F_y(y)\right) = 0. \qedhere
\end{align*}
\end{proof}

Two easy consequences can be stated.

\begin{cor}
If $\dim X = 2$, then $F$ is Gateaux differentiable in all
non-zero points.
\end{cor}

\begin{cor}
If $X$ is smooth (i.e.\ if every point of $S_X$ is smooth),
then $F$ is Gateaux differentiable in all non-zero points.
\end{cor}

Finally, we state the main result of the paper.

\begin{theo} \label{maintheo}
Let $X$ be an $m$-dimensional polyhedral space, $E$ a
finite-dimensional Banach space and $f:S_X\longrightarrow S_E$
a bijective isometry. Then, the homogeneous extension $F$ of
$f$ is a linear operator and, therefore, a linear isometry.
\end{theo}

\begin{proof}
It is shown in \cite[p.~377]{ting} (using Mankiewicz result
\cite{Mankiewicz}), that for every cone $C_j$ generated by an
$(m-1)$-dimensional face of $S_X$ there is a linear operator
$A_j$, such that $F(y) = A_j y$ for $y \in C_j$. In every
vertex, according to (2) of Theorem~\ref{diff-cond}, $F$ is
Gateaux differentiable, so all the $A_j$ that correspond to
faces that meet in this vertex are the same. This means that
all $A_j$ are the same linear operator $A$ and so $F = A$.
\end{proof}

\section{Concluding remarks}

From the Tingley's problem about bijective isometries of
spheres one can extract two weaker questions:
\begin{enumerate}
\item[(1)] If such an isometry exists, is it true that the
    corresponding spaces are isomorphic?
\item[(2)] If such an isometry exists, is it true that the
    corresponding spaces are isometric?
\end{enumerate}

Of course, the first question is meaningful only in the
infinite-dimensional case. Remark that, since the homogeneous
extension $F$ of the the bijective isometry $f: S_X
\longrightarrow S_E$ is a Lipschitz homeomorphism
\cite[Proposition 4.1]{cheng}, the question (1) is closely
related to a still open problem of whether Lipschitz
homeomorphism of separable Banach spaces implies linear
isomorphism. This problem have been studied by a number of
extraordinary mathematicians, and there are many deep and
interesting partial results (\cite{heinrich}, \cite{kal}).

The second question is quite interesting even for
finite-dimensional spaces. Our Lemma~\ref{lem 5} means, in
particular, that for a smooth space $X$ the existence of a
bijective isometry $f: S_X \longrightarrow S_E$ implies that
every 1-codimensional subspace of $X$ is isometric to a
1-codimensional subspace of $E$, and this correspondence
between 1-codimensional subspaces is bijective. If $\dim X \ge
3$, then this condition is quite restrictive and we wonder
whether it implies that $X$ and $E$ are isometric.

\bibliographystyle{amsplain}

\end{document}